\documentclass{amsart}
\usepackage{amssymb,amsthm,amsmath}

\newtheorem{thm}{Theorem}

\newtheorem{lem}[thm]{Lemma}
\newtheorem{cor}{Corollary}[thm]
\newtheorem{fact}[thm]{Fact}

\theoremstyle{definition}

\theoremstyle{remark}

\numberwithin{equation}{section}
\begin{document}

\title{An example on volumes jumping over Zariski dense set}
\author{Lue Pan}
\address{Lue Pan: Department of Mathematics, Princeton University, Fine Hall, Washington Road, Princeton, NJ 08540}
\email{lpan@princeton.edu}

\author{Junliang Shen}
\address{Junliang Shen: Departement Mathematik, ETH Z\"urich, R\"amistrasse 101, 8092 Z\"urich, Switzerland}
\email{junliang.shen@math.ethz.ch}

\thanks{The subject classification codes: 14C20, 14D99}
   
  \begin{abstract}
We give an example that the volume of an $\mathbb{R}$-divisor on a family of complex smooth surfaces jumps at infinite many prime divisors in the base. Our example follows the construction in \cite{1}.
  \end{abstract}
  
  \maketitle
  
   We work over the complex number $\mathbb{C}$. The volume of a line bundle $L$ on an irreducible projective variety $X$ of dimension $n$ is defined to be the nonnegative real number
  
\[
  {\rm vol}(L) = {\rm vol}_X(L) =  \limsup_{m \rightarrow \infty} \frac { h^0(X,L^{\otimes m})}{m^n/n!}
 \]
  For a Cartier divisor $D$, the volume $vol(D)$ is just the volume of the corresponding line bundle $\mathcal{O}_X(D)$. It is clear that ${\rm vol}(L) > 0$ if and only if $L$ is big. We use $N^1 (X)$ to denote the group of $\mathbb{R}$-divisors on X modulo numerical equivalence, the volume function can be uniquely extended as a continuous function on  $N^1 (X)$. More details could be found in \cite{2}. 
    
  Recall that by semicontinuity theorem, the volume of an $\mathbb{R}$-divisor on a family of varieties has to be constant over very general fibers. However, in most examples, the locus of volume's jump is Zariski closed. We will show that there exists an $\mathbb{R}$-divisor on a family of smooth surfaces obtained by  blow-up of $\mathbb{P}^2$, with volume jumping at a Zariski dense subset. The idea of construction comes from \cite{1}.
  
   Our main result is the following theorem:
  
 \begin{thm} \label{paper1}
 Let $\Sigma = (\mathbb{P}^2)^{10} \backslash \Delta$, where $\Delta$ is the locus where two points coincide and W is the subvariety of $\Sigma$ with the property that all these ten points lie on an elliptic curve. Suppose $\mathcal{X} \longrightarrow W$ is the family over $W$ whose fiber over $\mathbf{p}$ (written as $X_\mathbf{p}$) is the blow-up of $ \mathbb{P}^2$ at the corresponding ten points. There exists an $\mathbb{R}$-divisor $C$ on $\mathcal{X}$ such that $C_{ \mathbf{p}}$ is not big on $X_\mathbf{p}$ for very general $\mathbf{p}$, but there are countably many prime divisors $W_n \subset W$ such that $C_{\mathbf{p}}$ is big for general $\mathbf{p} \in W_n$.
 \end{thm} 
 
 By Theorem  \ref {paper1}, the volume of $C_{\mathbf{p}}$ is zero for very general $\mathbf{p}$, but the volume is positive for those general $\mathbf{p}$ in $W_n$. Hence it could serve as the example we want. 
 
 Since we only work on smooth surfaces, $N^1(X)$ is isomorphic to $N_1(X)$ which denotes the group of curves of coefficients in $\mathbb{R}$ modulo numerical equivalence. We use the same notation to present their elements. We review the construction in \cite{1} in the next section and then finish the proof of Theorem \ref{paper1}.

 \section{Lesieutre's construction}
We first recall John Lesieutre's construction(\cite{1} section 3) of a family of $\mathbb{R}$-divisors which is nef for very general fiber, but fails to be nef over countably many prime divisors in the base. Consider the family over $\Sigma$  whose fiber over $\mathbf{p} \in \Sigma $ is the blow-up of $ \mathbb{P}^2$ at the corresponding ten points. Then for any fiber $X=X_\mathbf{p} $, there are decompositions $N^1(X)=\mathbb{R}H\bigoplus_{i=1}^{10}\mathbb{R}E_i$. $H$ is the pullback of the hyperplane class on $\mathbb{P}^2$, and $E_i$ are the exceptional divisors.

John Lesieutre's idea is that he considers the birational map $\rho$ from $\Sigma$ to itself, given by $(p_1,  \cdots , p_{10}) \longmapsto (p_8, p_9, p_{10}, {\rm Cr}(p_1), \cdots, {\rm Cr}(p_7))$, where ${\rm Cr}$ is the Cremona transformation centered at $p_8, p_9, p_{10}$. Note that for general $\mathbf{p}\in\Sigma$, $\rho$ induces isomorphism between $X_\mathbf{p} $ and $X_{\rho(\mathbf{p} )}$, and also isomorphism from $N^1(X_\mathbf{p})$ to $N^1(X_{\rho(\mathbf{p})})$. Under the basis ${H,E_1,\cdots,E_{10}}$, it could be written as:
\[T = \mathbf{A} \cdot \mathbf{B}\]
\[\mathbf{A} = 
\left(\begin{array}{c|c}
M & 0\\
\hline
0 & I_7
\end{array}\right)
\]
$M$ is the transform matrix of Cremona transformation and  $\mathbf{B}$ is the permutation matrix for (1)(9, 10, 11, 2, 3, 4, 5, 6, 7, 8). 

It can be checked that $T$ has a unique eigenvalue $\lambda$ of magnitude greater than $1$. When its corresponding eigenvector $C_{\lambda}$ is written as $H-\sum_{i=1}^{10}r_iE_i$, we have $r_1+r_2+r_3>1$. \cite{1} Lemma 3.1 shows that for very general $\mathbf{p}\in\Sigma$, $C_{\lambda,\mathbf{p}}$ is nef on $X_\mathbf{p}$. Moreover, if we define $V_0\subset\Sigma$ to be the set of $\mathbf{p}$ for which first three points are colinear, then since the inequality $r_1+r_2+r_3>1$, $C_{\lambda}$ is not nef on $V_0$. Now, we define $V_{n+1}$ as the strict transform of $V_n$ under $\rho$. The discussion above shows that over $V_n$, $C_{\lambda}$ can not be nef. Since that $V_m$ and $V_n$ are distinct (\cite{1} Lemma 3.3) if $m\neq n$, $C_{\lambda}$ is not nef on countably prime divisors in $\Sigma$.

We list two important facts which we will use later.
\begin{fact}\label{fact1}
$C_{\lambda,\mathbf{p}}$'s self-intersection number is $0$, i.e. $C_{\lambda,\mathbf{p}}\cdot C_{\lambda,\mathbf{p}}=0$
\end{fact}
\begin{fact}\label{fact2}
$C_{\lambda,\mathbf{p}}\cdot K_{X_{\mathbf{p}}}=0$, where $K_{X_{\mathbf{p}}}=-3H+\sum_{i=1}^{10}E_i$ is the canonical divisor on $X_\mathbf{p}$.
\end{fact}
\begin{proof} Since $\rho$ induces isomorphism between $X_\mathbf{p}$ and $X_{\rho(\mathbf{p})}$, thus also $N^1(X_\mathbf{p})$ and $N^1(X_{\rho(\mathbf{p})})$, and the intersection form on them. That is to say:
\[
C_{\lambda,\mathbf{p}}\cdot C_{\lambda,\mathbf{p}}=T(C_{\lambda,\mathbf{p}})\cdot T(C_{\lambda,\mathbf{p}})
=\lambda^2C_{\lambda,\mathbf{p}}\cdot C_{\lambda,\mathbf{p}}
\]  
Our first conclusion comes from the observation that $\lambda>1$.

Notice that $T(K_{X_{\mathbf{p}}})=K_{\rho(\mathbf{p})}$, so $-3H+\sum_{i=1}^{10}E_i$ is invariant under $T$. Thus
\[C_{\lambda,\mathbf{p}}\cdot K_{X_{\mathbf{p}}}=T(C_{\lambda,\mathbf{p}})\cdot T(K_{X_{\mathbf{p}}})=\lambda C_{\lambda,\mathbf{p}}\cdot K_{X_{\mathbf{p}}}
\]
By $\lambda>1$, $C_{\lambda,\mathbf{p}}\cdot K_{X_{\mathbf{p}}}$ has to be zero.
\end{proof}

\section{Proof of Theorem \ref{paper1}}

Let $W'_0$ be the intersection of $W$ and $V_0$. Fix a point $\mathbf{p_0}=(p_1,\cdots,p_{10})\in W'_0$ by choosing $p_1,\cdots,p_{10}$ lie on an elliptic curve $E$, and first three points lie on a line $l$ which meets $E$ transversely, and $p_4,\cdots,p_{10}$ have the property that if $3d-\Sigma_{i=1}^{10}m_i=0$, the class $dl|_E-\Sigma_{i=1}^{10}m_ip_i$ is not linearly equivalently to $0$ on $E$ unless $m_4=\cdots=m_{10}=0$. This condition will be met for very general points in $W'_0$. We define $W_0$ as the irreducible component of $W'_0$ containing $\mathbf{p_0}$.

Define inductively $W_{n+1}$ as the strict transform of $W_n$ under $\rho|_W$. We want to show that $C_{\lambda,\mathbf{p}}$ has zero volume for very general point $\mathbf{p}\in W$, while for general $\mathbf{p}\in W_n$, the volume is non-trivial. But before that, we need to prove that $W_n$ is well-defined.

\begin{lem}Each $W_n$ is a well-defined prime divisor in $W$, and $W_m$ and $W_n$ is distinct if $n\ne m$.
\end{lem}

\begin{proof}The proof is almost the same as \cite{1} Lemma 3.3. Let $L$ be the divisor of $W$ where last three points are colinear. Recalling $\rho$'s definition, the Cremona transformation is centered at last three points. So we have to show $W_n$ are not contained in $L$. We prove the lemma by contructing a sequence of points $\mathbf{p_n}$, such that $\mathbf{p_n}\in W_n\setminus L$, but does not belong to $W_m,~m\ne n$.

Define inductively $\mathbf{p_{n+1}}$ as $\rho(\mathbf{p_n})$. Denote the strict transform of $l$ in $X_{\mathbf{p_0}}$ as $\bar{l}$, and also its class. It is clear that $X_{\mathbf{p_n}}$ has a rational curve in the class $T^n(\bar{l})$, and now we argue that this is the unique rational curve of self-intersection less than or equal to $-2$ on $X_{\mathbf{p_n}}$.

It suffices to prove the case $n=0$ since $\rho$ gives an isomorphism between $X_{\mathbf{p_{n+1}}}$ and $X_{\rho(\mathbf{p_n})}$. Suppose that $C\sim dH-\Sigma_{i=1}^{10}m_iE_i$ is a rational curve on $X_{\mathbf{p_0}}$ with $C^2\le -2$. Then by adjunction formula, $-2=C^2+C\cdot K_{X_{\mathbf{p_0}}}$. So $C\cdot K_{X_{\mathbf{p_0}}}\ge 0$. However, $-K_{X_{\mathbf{p_0}}}$ is the class of the strict transform of $E$, hence $C\cdot K_{X_{\mathbf{p_0}}}\le 0$. Thus $C\cdot K_{X_{\mathbf{p_0}}}$ must be $0$, i.e. $3d-\Sigma_{i=1}^{10}m_i=0$. The condition on the points imply immediately that $C$ has to be $\bar{l}$.

Caculating explicitly the eigenvalue of $T$, we find that there is no eigenvalue other than $1$ which is a root of unity. And it is easy to see $\bar{l}$ is not the eigenvector of $T$. Therefore the class $T^n(\bar{l})$ and $T^m(\bar{l})$ (view as an vector in $\mathbb{R}^{11}$) are distinct, if $n\ne m$. Since $X_{\mathbf{p}_n}$ has a unique rational curve with self-intersection less than $-2$, which lies in the class $T^n(\bar{l})$, $\mathbf{p_n}$ cannot belong to $W_m,~m\ne n$, and also $L$.
\end{proof}

By the same method as \cite{1} Lemma 3.4., it can be shown that $C_{\lambda,\mathbf{p}}$ is nef for very general $\mathbf{p}\in W$. Thus for these $\mathbf{p}$, ${\rm vol}_{X_\mathbf{p}}(C_{\lambda,\mathbf{p}})=C_{\lambda,\mathbf{p}}\cdot C_{\lambda,\mathbf{p}}=0$, thanks to the Fact 2 above.

The following lemma tells us that it suffices to prove ${\rm vol}_{X_\mathbf{p}}(C_{\lambda,\mathbf{p}})$ is not zero for general $\mathbf{p}\in W_0$.
\begin{lem}${\rm vol}_{X_\mathbf{p}}(C_{\lambda,\mathbf{p}})=\lambda^2 {\rm vol}_{X_\rho(\mathbf{p})}(C_{\lambda,\rho(\mathbf{p})})$, for any $\mathbf{p}\in\Sigma$ whose last three points are not colinear.
\end{lem}
\begin{proof}
$\rho$ induces isomorphism between $X_\mathbf{p}$ and $X_{\rho(\mathbf{p})}$, hence 
\[
{\rm vol}_{X_\mathbf{p}}(C_{\lambda,\mathbf{p}})={\rm vol}_{X_\rho(\mathbf{p})}(T(C_{\lambda,\mathbf{p}}))={\rm vol}_{X_\rho(\mathbf{p})}(\lambda C_{\lambda,\rho(\mathbf{p})})
=\lambda^2 {\rm vol}_{X_\rho(\mathbf{p})}(C_{\lambda,\rho(\mathbf{p})})
\]
\end{proof}

We use $\overline{\ell}$ to denote the divisor $H-E_1-E_2-E_3$, then it is effective on $X_\mathbf{p},~\mathbf{p}\in W_0$.
\begin{lem}
$C_{\lambda, \mathbf{p}} - \beta\overline{\ell}$ is nef for general $\mathbf{p} \in W_0$, where $\beta=\frac{C_{\lambda,\mathbf{p}}\cdot\overline{\ell}}{\overline{\ell}\cdot\overline{\ell}}=\frac{1}{2}(r_1+r_2+r_3-1)$.
 \end{lem}
 
 \begin{proof}
 \[L = _{def} \frac{1}{1-\beta}(C_{\lambda, \mathbf{p}} - \beta\overline{\ell}) = H - \sum_{i=1}^{10}t_iE_i\]
 For any curve class $C = dH - \sum_{i=0}^{10}a_iE_i$, it is sufficient to prove $L \cdot C \geqslant$ 0, i.e. 
 \begin{equation}\label{sjl6}
 d \geqslant \sum_{i=1}^{10}a_it_i
 \end{equation}
 If $C$ is just the strict transform of the elliptic curve containing these ten points, then $C =-K_{X_{\mathbf{p}}}=3H-\sum_{i=1}^{10}E_i,~\overline{\ell}\cdot C=0$.
 \[ (C_{\lambda, \mathbf{p}} - \beta\overline{\ell}) \cdot C =-C_{\lambda, \mathbf{p}} \cdot K_{X_{\mathbf{p}}} = 0\]
 Thus we can assume that $C$ is an irreducible curve different from this curve. In this case,  $K_{X_{\mathbf{p}}} \cdot C \leqslant 0$, because $-K_{X_{\mathbf{p}}}$ is just the curve class. Since $C \cdot K_{X_{\mathbf{p}}} + C \cdot C \geqslant -2$ due to adjunction formula, we get $C \cdot C \geqslant -2$, i.e.
\begin{equation}
 d^2+2 \geqslant \sum_{i=1}^{10}a_i^2
\end{equation}
 Also, we know that $\sum_{i=1}^{10}t_i^2 = 1 - \frac{2\beta^2}{{(1 - \beta)}^2}$ by calculating $L^2$ directly. Then using Cauchy-Swharz inequality:
 \[(\sum_{i=0}^{10}a_it_i)^2 \leqslant (\sum_{i=0}^{10}a_i^2) \cdot (\sum_{i=0}^{10}t_i^2) \leqslant (1 - \frac{2\beta^2}{{(1 - \beta)}^2})\cdot(d^2 +2)  \]
 When $d > 6$, the right side is less than $d^2$. We only need to check (\ref{sjl6}) for those curve class with $d = 1, 2, 3, 4, 5, 6$. 

$L$ is approximately:  $H-0.354E_1-0.342E_2-0.304E_3-0.371E_4-0.363E_5-0.336E_6-0.260E_7-0.253E_8-0.235E_9
-0.181E_{10}$.

When $d=1$, $C$ corresponds to a line on $\mathbb{P}^2$. For general $\mathbf{p}=(p_1,\cdots,p_{10})\in W_0$, we may assume no three points of $p_1,\cdots,p_{10}$ are colinear other than $p_1,p_2,p_3$, and $p_ip_j$ is not tangent to the ellptic curve passing through these points. So $C=\bar{l}$ or $C=H-E_i-E_j$. In the second case, (\ref{sjl6}) becomes $1\ge t_i+t_j$, and this can be checked using the numerical data above. In the first case, $L\cdot \bar{l}=0$ by our construction of $L$.

When $d=2$, the argument is quite similar. For general $\mathbf{p}\in W_0$, we may assume no six points of $p_1,\cdots,p_{10}$ lie on a quadratic curve, and if every quadratic curve passing through five of these points intersects the ellptic curve transversely. Then $C$ can be written as $H-\Sigma_{k=1}^5 E_{i_k}$, and (\ref{sjl6}) can be easily verified in this case.

When $3\le d \le 6$, we consider the restrictions on $C$: $C\cdot K_{\mathbf{p}} \le 0,~C\cdot C \ge -2$. Rewrite these inequalities in term of $a_i$, we get: 
\begin{equation}\label{sjl7}
d^2+2 \geqslant \sum_{i=1}^{10}a_i^2,~3d\ge \sum_{i=1}^{10}a_i
\end{equation}

It is easy to see there are only finitely many possibilities for the choice of $a_i$, so we can check (\ref{sjl6}) for all the $a_i$ satisfying (\ref{sjl7}). Since $t_4>t_5>t_1>t_2>t_6>t_3>t_7>t_8>t_9>t_{10}$ (see the approximate value of $L$), we may assume $a_4\ge a_5\ge a_1\ge a_2\ge a_6 \ge a_3\ge a_7\ge a_8\ge a_9\ge a_{10}$. Also we only need to consider the extreme case, which means if we substitute $a_{10}$ by $a_{10}+1$ while keeping other $a_i$ invariant, then (\ref{sjl7}) cannot hold. Under these conditions, all possibilities are listed below and we could see that (\ref{sjl6}) hold in each case.

\begin{tabular}{|l|l|l|l|l|l|l|l|l|l|l|l|}
\hline
 d &$a_1$ &$a_2$ &$a_3$ &$a_4$ &$a_5$ &$a_6$ &$a_7$ &$a_8$ & $a_9$ &$a_{10}$ & $d-\sum_{i=1}^{10}a_it_i $ \\\hline
3 &1&0&0&3&1&0&0&0&0&0&1.169\\\hline
3 &1&1&0&2&2&1&0&0&0&0&0.500\\\hline
3 &1&1&1&2&1&1&1&1&0&0&0.045\\\hline
3 &1&1&1&1&1&1&1&1&1&0&0.181\\\hline
4 &1&0&0&4&1&0&0&0&0&0&1.798\\\hline
4 &0&0&0&3&3&0&0&0&0&0&1.798\\\hline
4 &2&1&0&3&2&0&0&0&0&0&1.110\\\hline
4 &1&1&1&3&2&1&1&0&0&0&0.564\\\hline
4 &1&1&1&3&1&1&1&1&1&1&0.257\\\hline
4 &2&2&1&2&2&1&0&0&0&0&0.500\\\hline
4 &2&1&1&2&2&1&1&1&1&0&0.093\\\hline
5 &1&0&0&5&1&0&0&0&0&0&2.426\\\hline
5 &1&1&0&4&3&0&0&0&0&0&1.730\\\hline
5 &2&1&1&4&2&1&0&0&0&0&1.098\\\hline
5 &1&1&1&4&2&1&1&1&1&0&0.704\\\hline
5 &3&0&0&3&3&0&0&0&0&0&1.735\\\hline
5 &2&2&0&3&3&1&0&0&0&0&1.070\\\hline
5 &2&1&1&3&3&1&1&1&0&0&0.594\\\hline
5 &2&2&1&3&2&2&1&0&0&0&0.532\\\hline
5 &2&2&1&3&2&1&1&1&1&1&0.199\\\hline
5 &2&2&2&2&2&2&1&1&1&0&0.111\\\hline
6 &1&0&0&6&1&0&0&0&0&0&3.054\\\hline
6 &2&0&0&5&3&0&0&0&0&0&2.346\\\hline
6 &1&1&1&5&3&1&0&0&0&0&1.718\\\hline
6 &2&2&0&5&2&1&0&0&0&0&1.689\\\hline
6 &2&1&1&5&2&1&1&1&0&0&1.214\\\hline
6 &2&1&0&4&4&1&0&0&0&0&1.417\\\hline
6 &3&2&0&4&3&0&0&0&0&0&1.680\\\hline
6 &3&1&1&4&3&1&1&0&0&0&1.122\\\hline
6 &2&2&1&4&3&2&0&0&0&0&1.058\\\hline
6 &2&2&1&4&3&1&1&1&1&0&0.646\\\hline
6 &3&3&1&3&3&1&0&0&0&0&1.070\\\hline
6 &3&2&1&3&3&2&1&1&0&0&0.562\\\hline
6 &2&2&2&3&3&2&2&0&0&0&0.606\\\hline
6 &2&2&2&3&3&2&1&1&1&1&0.195\\\hline
\end{tabular}
\\
 \end{proof}
 
 \begin{cor}
 $L$ is big, and so is $C_{\lambda,\mathbf{p}}$, for general $\mathbf{p}\in W_0$.
 \end{cor}
 \begin{proof}
 $L^2>0$, thus it is big since it is nef on general fibers. $C_{\lambda,\mathbf{p}}$ is the sum of a big divisor and an effective divisor, so it must be big.
 \end{proof}

 \section{acknowledgement}
 
The authors want to thank intellectual merit of John Lesieutre for his work \cite{1}. This paper forms a part of authors' undergraduate theses written under the advisor Chenyang Xu. In particular, the authors would like to thank Xu for suggesting the question and for generously sharing his ideas related to this paper. 

Junliang Shen was supported by grant ERC-2012-AdG-320368-MCSK in the group of Pandharipande at ETH Z\"urich. Lue Pan was supported by First Year Fellowship in Natural Sciences and Engineering.

 \end{document}